\documentclass[12pt]{amsart}
\usepackage[margin=1.4in]{geometry}
\usepackage{graphicx,bm}
\usepackage{slashed}
\usepackage[breaklinks,colorlinks]{hyperref}
\usepackage{verbatim}
\usepackage{cases}
\usepackage{empheq}
\usepackage{color}
\numberwithin{equation}{section}

\newtheorem{thm}{Theorem}[section]

\newtheorem{cor}[thm]{Corollary}
\newtheorem{lem}[thm]{Lemma}

\newtheorem{prop}[thm]{Proposition}

\newtheorem{defn}[thm]{Definition}
\newtheorem{rmk}[thm]{Remark}

\newcommand{\pt}{\partial}
\DeclareMathOperator{\dive}{div}

\allowdisplaybreaks

\begin{document}

\title[Uniqueness of Dirac-harmonic maps]{Uniqueness of Dirac-harmonic maps from a compact surface with boundary}
\author[J\"urgen Jost]{J\"urgen Jost}
\address{Max Planck Institute for Mathematics in the Sciences, Inselstrasse 22, 04103 Leipzig, Germany}
\email{jost@mis.mpg.de}

\author[Jingyong Zhu]{Jingyong Zhu}
\address{Department of Mathematics, Sichuan University, Chengdu 610065, China}
\email{jzhu@scu.edu.cn}



\begin{abstract}
As a commutative version of the supersymmetric nonlinear sigma model, Dirac-harmonic maps from Riemann surfaces were introduced fifteen years ago. They are critical points of an unbounded conformally invariant functional involving two fields, a map from a Riemann surface into a Riemannian manifold and a section of a Dirac bundle which is the usual spinor bundle twisted with the pull-back of the tangent bundle of the target by the map. As solutions to a coupled nonlinear elliptic system, the existence and regularity theory of Dirac-harmonic maps has already received much attention, while the general uniqueness theory has not been established yet. For uncoupled Dirac-harmonic maps, the map components are harmonic maps. Since the uniqueness theory of harmonic maps from a compact surface with boundary is known, it is sufficient to consider the uniqueness of the spinor components, which are solutions to the corresponding boundary value problems for a nonlinear Dirac equation. In particular, when the map components belong to $W^{1,p}$ with $p>2$, the spinor components are uniquely determined by boundary values and map components. For coupled Dirac-harmonic maps, the map components are not harmonic maps. So the uniqueness problem is more difficult to solve. In this paper, we study the uniqueness problem on a compact surface with boundary. More precisely, we prove the energy convexity for weakly Dirac-harmonic maps from the unit disk with small energy. This yields the first uniqueness result about Dirac-harmonic maps from
a surface conformal to the unit disk with small energy and arbitrary boundary values.
\end{abstract}
\maketitle

\vspace{2em}
\section{Introduction}

Since the maximum principle does not apply to solutions of nonlinear elliptic systems, there are much fewer  tools available than for single elliptic equations for answering questions regarding the symmetry, the uniqueness and the regularity of these solutions. Rivi\`ere \cite{riviere2007conservation} succeeded in writing 2-dimensional conformally invariant nonlinear elliptic partial differential equations (harmonic map equation, prescribed mean curvature equations,..., etc.) in divergence form. These divergence-free quantities generalize to target manifolds without symmetries the well known conservation laws for weakly harmonic maps into homogeneous spaces. Using the conservation law, Rivi\`ere proved the (interior) continuity of any $W^{1,2}$ weak solution $u:B_1\to N\hookrightarrow\mathbb{R}^K$ to
\begin{equation}\label{antisymmetric form}
    -\Delta u=\Omega\cdot\nabla u
\end{equation}
with $\Omega\in L^2(B_1,so(K)\otimes\wedge^1\mathbb{R}^2)$. More precisely, this special form of the nonlinearity enabled Rivi\`ere to obtain a conservation law for this system, i.e.
\begin{equation}\label{jacobian}
    \dive(A\nabla u)=\nabla^\perp B\cdot\nabla u
\end{equation}
where $A\in W^{1,2}(B_1,{\rm GL}_K(\mathbb{R}))$ and $B\in W^{1,2}_0(B_1,M_K(\mathbb{R}))$. Here $\nabla$ and $\nabla^\perp$ are the gradient and the orthogonal gradient, respectively. This was accomplished via a technique that we call Rivi\`ere's gauge decomposition, see Section 3. The right-hand side of this new system \eqref{jacobian} lies in the Hardy space $\mathcal{H}^1$ by a result of Coifman, Lions, Meyer and Semmes \cite{coifman1993Compensated}. Moreover, using a Hodge decomposition argument, one can show that $u$ lies locally in $W^{2,1}$ which embeds into $C^0$ in two dimension. The key to this fact is a special ``compensation phenomenon" for Jacobian determinants (e.g. the right-hand side of \eqref{jacobian}), which was first observed
by Wente \cite{Wente1969existence} (see also Lemma \ref{wente lemma}). These Wente type estimates have many interesting applications, see e.g. \cite{coifman1993Compensated,helein2002Harmonic,riviere2007conservation,Wente1969existence}.

 Rivi\`ere's gauge decomposition also applies to the uniqueness problem, which is a very important and difficult issue for conformally invariant systems of partial differential equations. See \cite{lamm2013Estimates} and \cite{lin2014uniformity} for the harmonic map and its flow (see also \cite{wang2012harmonic} for a different approach), \cite{laurain2020Energy} for the bi-harmonic map and its flow. In both cases, the Rivi\`ere's gauge decomposition is very useful to derive the energy convexity, and some form of convexity is usually the reason for the uniqueness.

Motivated by the supersymmetric nonlinear sigma model from quantum field theory (see \cite{jost2009geometry}), Dirac-harmonic maps from Riemann surfaces (with a fixed spin structure) into Riemannian manifolds were introduced in \cite{chen2006dirac}. They  generalize harmonic maps and harmonic spinors. From the variational point of view, they are critical points of a conformally invariant  action functional, whose Euler-Lagrange equations constitute  an elliptic system coupling a harmonic map type equation with a nonlinear Dirac equation.

The existence and regularity theory of Dirac-harmonic maps already received much attention (see  \cite{ammann2013dirac,chen2015dirac,jost2018geometric, jost2019alpha,jost2021existence,jost2019short,wang2009regularity} and the references therein), while the uniqueness problem is still quite open. As we mentioned above, for harmonic maps, the uniqueness has a deep relationship with the energy convexity. When the target manifold has  nonpositive sectional curvature, the Dirichlet energy is always convex (see \cite{jost2017riemannian}). For a general target manifold, if the energy is sufficiently small, one can still get the energy convexity and the uniqueness as follows.

\begin{thm}\cite{colding2007Width}\cite{lamm2013Estimates}\label{harmonic}
There exists a constant $\varepsilon_0$ depending only on $N$ such that if $u,v\in W^{1,2}(B_1,N)$ with
\begin{equation*}
E(u)=\int_{B_1}|\nabla u|^2\leq\varepsilon_0, \  \  u|_{\partial B_1}=v|_{\partial B_1}
\end{equation*}
and $u$ is a weakly harmonic map, then we have the energy convexity
\begin{equation*}
\frac12\int_{B_1}|\nabla v-\nabla u|^2\leq\int_{B_1}|\nabla v|^2-\int_{B_1}|\nabla u|^2.
\end{equation*}
Moreover, if $v$ is also a weakly harmonic map with
\begin{equation*}
E(v)=\int_{B_1}|\nabla v|^2\leq\varepsilon_0,
\end{equation*}
then we have $u=v$ in $B_1$. Here $B_1\subset\mathbb{R}^2$ is a disk.
\end{thm}

 In the context of Dirac-harmonic maps, we do not have the nonnegativity of the second variation formula (see \cite{jost2018geometric}) even if the target is a nonpositively curved manifold. Moreover, the action functional for Dirac-harmonic  maps is unbounded from below because of the Dirac term in the functional. Also, because of the Dirac operator, the right-hand side of the equation for the map component is no longer zero. It is a quadratic term of the spinor involving the curvature of the target manifold. For coupled Dirac-harmonic maps (i.e. when the map component is not a harmonic map), this curvature term does not vanish and we do not know what the map is. In particular, when the energy of the coupled Dirac-harmonic is small, this curvature term is a dominant term and generally controlled by good terms with a bad power. These issues make the uniqueness of coupled Dirac-harmonic maps more difficult to deal with.

 Different from  harmonic maps, we have to deal with the uniqueness of two components, both the maps and the spinor fields. For the spinor fields, the only uniqueness result so far was proved in \cite{chen2017estimates} for $W^{1,2\alpha}$ maps with $\alpha>1$. Together with the uniqueness of harmonic maps into target manifolds with nonpositive sectional curvatures and Theorem \ref{harmonic}, we observe the following uniqueness theorem for uncoupled Dirac-harmonic maps (i.e. when the map component is a harmonic map).

\begin{thm}\label{uncoupled dh thm}
Given any uncoupled Dirac-harmonic maps $(\phi_1,\psi_1)$, $(\phi_2,\psi_2)$ with $\phi_i\in W^{1,2\alpha}(B_1,N)$ and $\psi_i\in W^{1,4/3}(B_1,\Sigma B_1\otimes\phi_i^*TN)$ for some $\alpha>1$ and $i=1,2$, if the target manifold $N$ has anonpositive sectional curvature, or
\begin{equation}
    E(\phi_i,\psi_i):=\int_{B_1}|\nabla\phi_i|^2\leq\varepsilon_0, \forall  i=1,2
\end{equation}
for a constant $\varepsilon_0$ only depending on $N$ and
\begin{equation}\label{bdy values}
\begin{cases}
  \phi_1=\phi_2,\ & \text{on} \ \partial B_1,\\
  {\bf B}\psi_1={\bf B}\psi_2, \ & \text{on} \ \partial B_1,
    \end{cases}
\end{equation}
then $(\phi_1,\psi_1)=(\phi_2,\psi_2)$ in $B_1$. Here, ${\bf B}$ is the chiral boundary operator for the twisted Dirac operator, see the definition \eqref{bdy op}.
\end{thm}

A similar super-critical restriction on the regularity of the other component is also important when we investigate the uniqueness of the map component. By assuming the spinor fields lie in the space $W^{1,r}$ with $r>4/3$ and the energy $E(\phi_1,\psi_1)$ of the weakly Dirac-harmonic map $(\phi_1,\psi_1)$ is sufficiently small, we first prove the following energy convexity.

\begin{thm}\label{energy convexity}
Suppose that $(\phi_1,\psi_1)$ is a weakly coupled Dirac-harmonic map with $\phi_1\in W^{1,2}(B_1,N)$ and $\psi_1\in W^{1,r}(B_1,\Sigma B_1\otimes\phi_1^*TN)$ for some $r>4/3$. For any $\phi_2\in W^{1,2}(B_1,N)$ satisfying
\begin{equation}\label{1 dh}
    \int_{B_1}|\phi_2-\phi_1|^2|\nabla\phi_1|^2\geq c_0^2
\end{equation} for some positive constant $c_0$. There exists a constant $\varepsilon_0$ depending on the target manifold $N$ and $\|\psi_1\|_{W^{1,r}}$ such that
if
\begin{equation}\label{small energy}
    E(\phi_1):=\int_{B_1}|\nabla\phi_1|^2\leq\varepsilon_0, \ \ E(\psi_1):=\int_{B_1}|\psi_1|^4\leq\varepsilon_0^p, \ \forall i=1,2
\end{equation}
for some constant $p=p(c_0)\geq1$ determined by \eqref{p}
and
$\phi_1=\phi_2$ on $\partial B_1$,
then
\begin{equation}\label{energy convex}
    \frac12\int_{B_1}|\nabla \phi_2-\nabla{\phi_1}|^2\leq\int_{B_1}|\nabla{\phi_2}|^2-\int_{B_1}|\nabla{\phi_1}|^2.
\end{equation}
\end{thm}

Here, we give a remark about our assumptions.
\begin{rmk}\label{technique}
  First, we prove the energy convexity by showing the following
\begin{equation}
\begin{split}
 \Phi&:=\int_{B_1}|\nabla\phi_2|^2-\int_{B_1}|\nabla\phi_1|^2-\int_{B_1}|\nabla\phi_2-\nabla\phi_1|^2\\
 &\geq -\frac12\int_{B_1}|\nabla{\phi_2}-\nabla{\phi_1}|^2.
 \end{split}
\end{equation}
To estimate the left-hand side of the inequality above, we divide $\Phi$ into two terms by using the equation of $\phi_1$. One is from the second fundamental form, the other comes from the curvature, which is non-zero for coupled Dirac-harmonic maps. The first term is a good term and can be controlled by $\varepsilon_0\int_{B_1}|\nabla{\phi_2}-\nabla{\phi_1}|^2$ as in the harmonic map case, while the second term from the curvature is a bad term. In general, we can only control it by $\varepsilon_0(\int_{B_1}|\nabla{\phi_2}-\nabla{\phi_1}|^2)^{1/2}$. Therefore, when the energy of $\phi_2$ is also small, this bad term can be bigger than the first term. In this case, this bad term is a dominant term. To control it, we need the assumption \eqref{1 dh}. Under this assumption, the bad term is less than the first term multiplied by $c_0^{-1}$. Therefore, the bad term is controlled by $c_0^{-1}\varepsilon_0^{1+\frac{p}{2}}\int_{B_1}|\nabla{\phi_2}-\nabla{\phi_1}|^2$, and we can make the coefficient smaller than $\frac{1}{2}$ by choosing a suitable $p$.

Second, the condition $r>4/3$ is used to give the estimate of $L^\infty$ norm for $B$, see Proposition \ref{L infty B}. When $r=4/3$, the curvature term $\mathcal{R}(\phi_1,\psi_1)$ (see \eqref{curvature term}) only belongs to $L^1$, and it is well-known that in general the solution to the following equation
\begin{equation*}
    \Delta u=f\in L^1(B_1)
\end{equation*}
 does not even belong to the $H^1$ space.
\end{rmk}

Now, with the convexity \eqref{energy convex} in hands, it is natural to prove the uniqueness of coupled Dirac-harmonic maps by assuming $\phi_2$ is also a coupled Dirac-harmonic map. As a direct consequence of the above energy convexity, we have the following statement.

\begin{cor}\label{main thm}
Given any two weakly coupled Dirac-harmonic maps $(\phi_1,\psi_1)$, $(\phi_2,\psi_2)$ with $\phi_i\in W^{1,2}(B_1,N)$ and $\psi_i\in W^{1,r}(B_1,\Sigma B_1\otimes\phi_i^*TN)$ for some $r>4/3$ and
\begin{equation}\label{dhs}
    \int_{B_1}|\phi_1-\phi_2|^2|\nabla\phi_i|^2\geq c_0^2, \ \forall  i=1,2
\end{equation} for some positive constant $c_0$.
There exists a constant $\varepsilon_0$ depending on the target manifold $N$ and $\|\psi_i\|_{W^{1,r}}$ such that
if
\begin{equation}\label{spinor energy p}
    E(\phi_i):=\int_{B_1}|\nabla\phi_i|^2\leq\varepsilon_0, \ \ E(\psi_i):=\int_{B_1}|\psi_i|^4\leq\varepsilon_0^p, \ \forall i=1,2
\end{equation}
for some constant $p=p(c_0)\geq1$ determined by \eqref{p} and
$\phi_1=\phi_2$ on $\partial B_1$,
then $\phi_1=\phi_2$ in $B_1$.
\end{cor}

Notice that the conclusion is in contradiction with the condition \eqref{dhs}. This reminds us to prove  the following estimate on the difference between the underlying maps by contradiction.

\begin{cor}\label{main cor}
Given any two weakly coupled Dirac-harmonic maps $(\phi_1,\psi_1)$, $(\phi_2,\psi_2)$ satisfying $\phi_i\in W^{1,2}(B_1,N)$, $\phi_1=\phi_2$ on $\partial B_1$ and $\psi_i\in W^{1,r}(B_1,\Sigma B_1\otimes\phi_i^*TN)$ for some $r>4/3$. 
There exists a constant $\varepsilon_0$ depending on the target manifold $N$ and $\|\psi_i\|_{W^{1,r}}$ such that
if the small energy condition \eqref{spinor energy p} holds for some constant $p\geq1$, then we have
\begin{equation}
    \int_{B_1}|\phi_1-\phi_2|^2|\nabla\phi_j|^2<c_0\leq \varepsilon_0^p, \ \forall   i=1,2,
\end{equation}where $c_0$ is determined by \eqref{p}.
\end{cor}

Notice that the boundary values of the spinor fields may not be the same in our result.  Moreover, if we assume $\phi_i\in W^{1,2\alpha}$, we can remove the extra assumption on the regularity of the spinor fields.

\begin{thm}\label{unique DH}
Given any two coupled Dirac-harmonic maps $(\phi_1,\psi_1)$, $(\phi_2,\psi_2)$ with $\phi_i\in W^{1,2\alpha}(B_1,N)$ and $\psi_i\in W^{1,4/3}(B_1,\Sigma B_1\otimes\phi_i^*TN)$ for some $\alpha>1$, there exists a constant $\varepsilon_0$ only depending on the target manifold $N$  such that
if the small energy condition \eqref{spinor energy p} holds for some constant $p\geq1$ and the boundary values of the maps coincide, then we have
\begin{equation}
    \int_{B_1}|\phi_1-\phi_2|^2|\nabla\phi_j|^2< \varepsilon_0^p, \ \forall i=1,2.
\end{equation} 
\end{thm}

As we mentioned in the Remark \ref{technique}, there is no bad term when the Dirac-harmonic map is uncoupled. Therefore, the $(\phi_i,\psi_i)$ in Corollary \ref{main cor} and Theorem \ref{unique DH} can be two arbitrary Dirac-harmonic maps with the same boundary values of the maps. Moreover, when the integral $  \int_{B_1}|\phi_1-\phi_2|^2|\nabla\phi_i|^2$ is equal to zero, then the $\Phi$ in Remark \ref{technique} is nonnegative. Therefore, the energy convexity is also true. Hence, we can combine  Corollary \ref{main cor} and Theorem \ref{unique DH}  into the following statement.

\begin{thm}\label{main thm'}
Given two  Dirac-harmonic maps $(\phi_1,\psi_1)$, $(\phi_2,\psi_2)$ satisfy $\phi_1=\phi_2$ on $\partial B_1$, assume $\phi_i\in W^{1,2}(B_1,N)$ and $\psi_i\in W^{1,r}(B_1,\Sigma B_1\otimes\phi_i^*TN)$ for some $r>4/3$ (or $\phi_i\in W^{1,2\alpha}(B_1,N)$ and $\psi_i\in W^{1,4/3}(B_1,\Sigma B_1\otimes\phi_i^*TN)$ for some $\alpha>1$).
There exists a constant $\varepsilon_0$ depending on the target manifold $N$ and $\|\psi_i\|_{W^{1,r}}$ (or only depending on $N$) such that
if the small energy condition \eqref{spinor energy p} holds for some $p\geq1$, then we have either $\phi_1=\phi_2$ on $B_1$ or
\begin{equation}\label{latter case}
    0<\int_{B_1}|\phi_1-\phi_2|^2|\nabla\phi_j|^2< \varepsilon_0^p, \ \forall i=1,2.
\end{equation} 
\end{thm}

By the compactness of Dirac-harmonic maps, we can exclude the case in \eqref{latter case} and prove the following uniqueness result.

\begin{thm}\label{main uniqueness}
Let $i=1,2$. Given a boundary condition $(\varphi,{\bf B}\psi_0)$ with ${\bf B}\psi_0\not\equiv0$ and Dirac-harmonic map $(\phi_i,\psi_i)$ with $\phi_i\in W^{1,2\alpha}(B_1,N)$, $\psi_i\in W^{1,4/3}(B_1,\Sigma B_1\otimes\phi_i^*TN)$ for some $\alpha>1$.
There exists a constant $\varepsilon_0$ only depending on the target manifold $N$ and a constant $p>1$ such that if $(\phi_i,\psi_i)$ satisfy the small energy condition \eqref{spinor energy p} for $p_0$ and the following boundary condition
 \begin{equation}
\begin{cases}
  \phi_1=\phi_2=\varphi,\ & \text{on} \ \partial B_1,\\
  {\bf B}\psi_1={\bf B}\psi_2={\bf B}\psi_0, \ & \text{on} \ \partial B_1,
    \end{cases}
\end{equation}
then we have $(\phi_1,\psi_1)=(\phi_2,\psi_2)$ in $B_1$. 

Moreover, there is a universal constant $p_0$ such that $p\leq p_0$ unless there is a sequence of coupled Dirac-harmonic maps converging to the Dirac-harmonic map $(\phi_0,0)$, where $\phi_0$ is the unique harmonic map with boundary value $\varphi$.

\end{thm}

\begin{rmk}
For the case of $\phi\in W^{1,2}(B_1,N)$,  we can also prove the uniqueness of the map components in the same way as Theorem \ref{main uniqueness}. However, the uniqueness of the spinor fields is not known. This is the reason why we cannot write down a similar uniqueness result in this paper. 
\end{rmk}

The rest of the paper is organized as follows: In Section 2 and Section 3, we recall some facts about Dirac-harmonic maps and  Rivi\`ere's gauge decomposition, respectively. In Section 4 and Section 5, we prove two estimates on two important quantities given by Rivi\`ere's gauge decomposition. In Section 6, we prove the Theorem \ref{energy convexity} and Corollary \ref{main thm}. In Section 7, we prove the Theorem \ref{unique DH} and Theorem \ref{main uniqueness}.

{\bf Acknowledgements}
The second author would like to thank the support by the National Natural Science Foundation of China (Grant No. 12201440) and the Fundamental Research Funds for the Central Universities.
 Our  gratitude goes to the anonymous reviewers for their careful work and thoughtful suggestions that have helped to improve this paper.

\vspace{2em}

\section{Preliminaries}

Let $(M, g)$ be a compact Riemann surface with a fixed spin structure. On the complex spinor bundle $\Sigma M$, we denote the Hermitian inner product by $\langle\cdot, \cdot\rangle_{\Sigma M}$. For any $X\in\Gamma(TM)$ and $\xi\in\Gamma(\Sigma M)$, the Clifford multiplication satisfies the following skew-adjointness:
\begin{equation*}
\langle X\cdot\xi, \eta\rangle_{\Sigma M}=-\langle\xi, X\cdot\eta\rangle_{\Sigma M}.
\end{equation*}
Let $\nabla$ be the Levi-Civita connection on $(M,g)$. There is a unique  connection (also denoted by $\nabla$) on $\Sigma M$ compatible with $\langle\cdot, \cdot\rangle_{\Sigma M}$.  Choosing a local orthonormal basis $\{e_{\beta}\}_{\beta=1,2}$ on $M$, the usual Dirac operator is defined as $\slashed\partial:=e_\beta\cdot\nabla_\beta$, where $\beta=1,2$. Here and in the sequel, we use the Einstein summation convention. One can find more about spin geometry in \cite{lawson1989spin}.

Let $\phi$ be a smooth map from $M$ to another compact Riemannian manifold $(N, h)$ of dimension $n\geq2$. Let $\phi^*TN$ be the pull-back bundle of $TN$ by $\phi$ and consider the twisted bundle $\Sigma M\otimes \phi^*TN$. On this bundle there is a metric $\langle\cdot,\cdot\rangle_{\Sigma M\otimes \phi^*TN}$ induced from the metric on $\Sigma M$ and $\phi^*TN$. Also, we have a connection $\tilde\nabla$ on this twisted bundle naturally induced from those on $\Sigma M$ and $\phi^*TN$. In local coordinates $\{y^\mu\}_{\mu=1,\dots,n}$, the section $\psi$ of $\Sigma M\otimes \phi^*TN$ is written as
$$\psi=\psi^\mu\otimes\partial_{y^\mu}(\phi),$$
where each $\psi^\mu$ is a usual spinor on $M$. We also have the following local expression of $\tilde\nabla$
$$\tilde\nabla\psi=\left(\nabla\psi^\mu+\Gamma_{\lambda\sigma}^\mu(\phi)\nabla \phi^\lambda\cdot\psi^\sigma\right)\otimes\partial_{y^\mu}(\phi),$$
where $\Gamma^\mu_{\lambda\sigma}$ are the Christoffel symbols of the Levi-Civita connection of $N$. The Dirac operator along the map $\phi$ is defined as
\begin{equation}\label{dirac}
\slashed{D}\psi:=e_\alpha\cdot\tilde\nabla_{e_\alpha}\psi=\left(\slashed\partial\psi^\mu+\Gamma_{\lambda\sigma}^\mu(\phi)\nabla_{e_\alpha}\phi^\lambda(e_\alpha\cdot\psi^\sigma)\right)\otimes\partial_{y^\mu}(\phi),
\end{equation}
which is self-adjoint (see \cite{jost2017riemannian}). Sometimes, we use $\slashed{D}^\phi$ to distinguish the Dirac operators along different maps. In \cite{chen2006dirac}, the authors  introduced the  functional
\begin{equation*}\begin{split}
L(\phi,\psi)&:=\frac12\int_M\left(|d\phi|^2+\langle\psi,\slashed{D}\psi\rangle_{\Sigma M\otimes\phi^*TN}\right)\\
&=\frac12\int_M\left( h_{\lambda\sigma}(\phi)g^{\alpha\beta}\frac{\partial \phi^\lambda}{\partial x^\alpha}\frac{\pt \phi^\sigma}{\pt x^\beta}+h_{\lambda\sigma}(\phi)\langle\psi^\lambda,\slashed{D}\psi^\sigma\rangle_{\Sigma M}\right).
\end{split}
\end{equation*}
They computed the Euler-Lagrange equations of $L$:
\begin{numcases}{}
   \tau^\kappa(\phi)-\frac12R^\kappa_{\lambda\mu\sigma}\langle\psi^\mu,\nabla \phi^\lambda\cdot\psi^\sigma\rangle_{\Sigma M}=0,  \label{eldh1}\\
   \slashed{D}\psi^\mu:=\slashed\partial\psi^\mu+\Gamma_{\lambda\sigma}^\mu(\phi)\nabla_{e_\alpha}\phi^\lambda(e_\alpha\cdot\psi^\sigma)=0, \label{eldh2}
\end{numcases}
where $\tau^\kappa(u)$ is the $\kappa$-th component of the tension field \cite{jost2017riemannian} of the map $\phi$ with respect to the coordinates on $N$, $\nabla\phi^\lambda\cdot\psi^\sigma$ denotes the Clifford multiplication of the vector field $\nabla\phi^\lambda$ with the spinor $\psi^\sigma$, and $R^\kappa_{\lambda\mu\sigma}$ stands for the components of the Riemann curvature tensor of the target manifold $N$. By denoting
\begin{equation}\label{curvature term}
    \mathcal{R}(\phi,\psi):=\frac12R^\kappa_{\lambda\mu\sigma}\langle\psi^\mu,\nabla \phi^\lambda\cdot\psi^\sigma\rangle_{\Sigma M}\pt_{y^\kappa},
\end{equation}
 we can write \eqref{eldh1} and \eqref{eldh2} in the following global form:
\begin{numcases}{}
\tau(\phi)=\mathcal{R}(\phi,\psi), \label{geldh1} \\
\slashed{D}\psi=0,  \label{geldh2}
\end{numcases}
and call the solutions $(\phi,\psi)$ Dirac-harmonic maps from $M$ to $N$.

By \cite{nash1956imbedding}, we can isometrically embed $N$ into a Euclidean space $\mathbb{R}^K$ for some constant $K$. Then \eqref{geldh1}-\eqref{geldh2} is equivalent to the  following system:
\begin{numcases}{}
		\Delta_g{\phi}=II(d\phi,d\phi)+Re(P(\mathcal{S}(d\phi(e_\beta),e_{\beta}\cdot\psi);\psi)), \label{map eq}\\
		\slashed{\partial}\psi=\mathcal{S}(d\phi(e_\beta),e_{\beta}\cdot\psi).
\end{numcases}
Here $II$ is the second fundamental form of $N$ in $\mathbb{R}^K$,
\begin{equation*}
\mathcal{S}(d\phi(e_\beta),e_{\beta}\cdot\psi):=(\nabla{\phi^i}\cdot\psi^j)\otimes II(\partial_{z^i},\partial_{z^j}),
\end{equation*}
\begin{equation*}
Re(P(\mathcal{S}(d\phi(e_\beta),e_{\beta}\cdot\psi);\psi)):=P(S(\partial_{z^k},\partial_{z^j});\partial_{z^i})Re(\langle\psi^i,d\phi^k\cdot\psi^j\rangle).
\end{equation*}
where $i,j,k=1,\dots,K$, $P(\xi;\cdot)$ denotes the shape operator, defined by
\begin{equation*}
  \nabla_X\xi=-P(\xi;X)+\nabla^\perp\xi, \forall X\in \Gamma(TN)
\end{equation*}
and $Re(\zeta)$ denotes the real part of $\zeta\in\mathbb{C}$ .

When $M$ has a boundary, it is natural to propose a well-defined  boundary condition for Dirac-harmonic maps. For the map component, one can still use the Dirichlet boundary condition. For the spinor part, the chiral boundary operator ${\bf B}={\bf B}^\pm$ was first introduced in \cite{chen2013boundary} as follows:
\begin{equation}\label{bdy op}
\begin{split}
{\bf B}^\pm: L^2(\partial M, (\Sigma M\otimes\phi^*TN)|_{\partial M})&\to L^2(\partial M, (\Sigma M\otimes\phi^*TN)|_{\partial M})\\
\psi&\mapsto\frac12(Id\otimes Id\pm{\bf n}\cdot G\otimes Id)\psi
\end{split}\end{equation}
where ${\bf n}$ is the outward unit normal vector field on $\partial M$ and $G$ is the chirality operator for usual spinors which was firstly introduced by Gibbons-Hawking-Horowitz-Perry \cite{Gibbons1983Positive} and satisfies
\begin{equation}
G^2=Id, \ G^*=G, \ \nabla G=0, G\cdot X\cdot=-X\cdot G, \forall X\in\Gamma(TM).
\end{equation}

With the aim to get a general existence scheme for  Dirac-harmonic maps, the short-time existence of the following heat flow for Dirac-harmonic maps was proved in \cite{chen2017estimates}:
\begin{equation}
\begin{cases}
	\partial_t\phi=\tau(\phi)-\mathcal{R}(\phi,\psi),\ & \text{on}\   (0,T)\times M, \\
		\slashed{D}\psi=0,\ & \text{on} \   [0,T]\times M.	
\end{cases}
\end{equation}
Such a flow may develop singular points, where one can get the bubbles. In general, these bubbles could be Dirac-harmonic maps from spheres and nobody knows how to exclude them. To avoid this issue, the authors in \cite{jost2018geometric} used the $\alpha$-Dirac-harmonic map flow instead and solved the boundary value problem for Dirac-harmonic maps.

At the same time, we successfully applied $\alpha$-Dirac-harmonic map flow to the existence problem of Dirac-harmonic maps from closed surfaces, see our paper \cite{jost2019short}. In this case, the difficulty is that we can not uniquely solve the Dirac equation. To overcome this, we assume the Dirac operator along the initial map has minimal dimensional kernel as in \cite{wittmann2017short} to get the short-time existence of the $\alpha$-Dirac-harmonic map flow. However, the flow may still develop singularity, where the dimension of the kernel of the Dirac operator increases along the flow. This is different from the compact case. At those singular maps, the corresponding Dirac operators have high dimensional kernels which may split in general. In this case, the long-time existence of the flow is hard to achieve. Instead, we use the density of the maps along which the Dirac operators have minimal dimensional kernels in the homotopy class to restart the flow whenever we approach a singular point. In this way, we can get the existence of $\alpha$-Dirac-harmonic maps from closed surfaces. Then the existence of Dirac-harmonic maps follows from the standard blow-up analysis.

\vspace{2em}

\section{Rivi\`ere's gauge decomposition}
In \cite{riviere2007conservation}, Rivi\`ere used the algebraic feature of $\Omega$, namely $\Omega$ being antisymmetric, to construct $\xi\in W^{1,2}_0(B_1, so(K))$ and a gauge transformation matrix $P\in W^{1,2}\cap L^{\infty}(B_1, SO(K))$ (which pointwisely almost everywhere is an orthogonal matrix in $\mathbb{R}^{K\times K}$) satisfying some good properties.

\begin{thm}\cite[Lemma A.3]{riviere2007conservation}\label{xip}
There exists $\varepsilon>0$ and $C>0$ such that for every $\Omega\in L^2(B_1, so(K)\otimes\wedge^1\mathbb{R}^2)$ satisfying
\begin{equation*}
    \int_{B_1}|\Omega|^2\leq\varepsilon,
\end{equation*}
there exists $\xi\in W^{1,2}_0(B_1,so(K))$ and $P\in W^{1,2}(B_1,SO(K))$ such that
\begin{numcases}{}
	\nabla^\perp{\xi}=P^T\nabla P+P^T\Omega P \   \text{in} \   B_1, \label{xi} \\
	\xi=0\     \text{on} \   \partial B_1,	
\end{numcases}
and
\begin{equation}
    \|\nabla{\xi}\|_{L^2(B_1)}+\|\nabla{P}\|_{L^2(B_1)}\leq C\|\Omega\|_{L^2(B_1)}.
\end{equation}
Here the superscript $T$ denotes the transpose of a matrix.
\end{thm}

 Another important result from Rivi\`ere's work is the following theorem.

 \begin{thm}\cite[Theorem I.4]{riviere2007conservation}\label{AB}
 There exists $\varepsilon>0$ and $C>0$ such that for every $\Omega\in L^2(B_1, so(K)\otimes\wedge^1\mathbb{R}^2)$ satisfying
\begin{equation*}
    \int_{B_1}|\Omega|^2\leq\varepsilon,
\end{equation*}
there exists
\begin{equation*}
    \begin{split}
        \hat{A}&\in W^{1,2}\cap C^0(B_1,{\rm GL}_{K}(\mathbb{R})),\\
        A=(\hat{A}+Id)P^T&\in W^{1,2}\cap L^{\infty}(B_1,{\rm GL}_{K}(\mathbb{R})),\\
        B&\in W^{1,2}_0(B_1, M_{K}(\mathbb{R}))
    \end{split}
\end{equation*}
such that
\begin{equation}\label{nabla omega}
    \nabla A-A\Omega=\nabla^\perp B
\end{equation}
and
\begin{equation}
    \|\hat{A}\|_{W^{1,2}(B_1)}+\|\hat{A}\|_{L^{\infty}(B_1)}+\|B\|_{W^{1,2}(B_1)}\leq C\|\Omega\|_{L^2(B_1)}.
\end{equation}
 \end{thm}

In \cite{chen2013boundary}, the authors rewrite \eqref{map eq} in the following form
\begin{equation}\label{rewrite map eq}
\begin{split}
    -\Delta\phi^m&=\Omega^m_i\cdot\nabla\phi^i\\
    &=(f^m_{il}\nabla\phi^l+\omega^m_i)\cdot\nabla\phi^i
    \end{split}
\end{equation}
 with $\Omega=(\Omega^m_i)_{1\leq i,m\leq K}\in L^2(B_1, so(K)\otimes\wedge^1\mathbb{R}^2)$ and $\omega^m_i$ is a second order term of $\psi$ involving the curvature of the target manifold. Moreover, it is easy to see that
 \begin{equation}
     \int_{B_1}|\Omega|^2\leq C(N)\left(\int_{B_1}|\nabla\phi|^2+|\psi|^4\right)=C(N)E(\phi,\psi).
 \end{equation}
 Combining \eqref{rewrite map eq} and \eqref{nabla omega}, we have the conservation law
\begin{equation}\label{conservation law}
    \dive(A\nabla\phi+B\nabla^\perp\phi)=0.
\end{equation}

\vspace{2em}

\section{Global estimate on the matrix $B$}
In this section, we will show that $B$ is close to zero matrix if $E(\phi,\psi)\leq\varepsilon_0$ is sufficiently small and $\psi_1\in W^{1,r}$ for some $r>4/3$. This is the only place we can see the dependence of $\varepsilon_0$ on the super-critical regularity of the spinor field. Before proving the estimate, let us recall the Wente's lemma which will be used later.

\begin{lem}\cite{Wente1969existence}\label{wente lemma}
Let $a,b\in W^{1,2}(B_1,\mathbb{R})$ and let $u$ be the solution of
\begin{equation}
\begin{cases}
\Delta u=\nabla a\cdot\nabla^\perp b, \  & \text{in} \ B_1,\\
u=0, \  & \text{on} \ \partial B_1.
\end{cases}
\end{equation}
Then $u\in C^0\cap W^{1,2}(B_1,\mathbb{R})$ and the following estimate holds:
\begin{equation}
    \|u\|_{L^\infty(B_1)}+\|\nabla u\|_{L^2(B_1)}\leq C\|\nabla a\|_{L^2(B_1)}\|\nabla b\|_{L^2(B_1)}.
\end{equation}
\end{lem}

Now, let us state the main estimate in this section.

\begin{prop}\label{L infty B}
Under the assumptions of Theorem \ref{energy convexity} and let $r=\frac{4s}{4-s}$, we have
\begin{equation}\label{B}
    \|B\|_{L^\infty(B_1)}\leq C\varepsilon_0^{1/4},
\end{equation}
where $C=C\left(N,\|\psi_1\|_{L^{\frac{4s}{2-s}}(B_1)},\|\nabla\psi_1\|_{L^{\frac{4s}{4-s}}(B_1)}\right)$.
\end{prop}

\begin{proof}
Taking the curl on both sides of equation \eqref{nabla omega} yields
\begin{equation}
\begin{split}
    \Delta B^i_j&=-{\rm curl}(A^i_k\Omega^k_j)\\
    &=-{\rm curl}(A^i_k(f^k_{jl}\nabla\phi_1^l+\omega^k_j))\\
    &=-\nabla^\perp(A^i_kf^k_{jl})\cdot\nabla\phi_1^l-{\rm curl}(A^i_k\omega^k_j).
    \end{split}
\end{equation}
Now, we let $F^i_j,G^i_j\in W^{1,2}_0(B_1)$ be the solutions of
\begin{equation}
    \begin{split}
        \Delta F^i_j&=-\nabla^\perp(A^i_kf^k_{jl})\cdot\nabla\phi^l_1\\
        \Delta G^i_j&=-{\rm curl}(A^i_k\omega^k_j),
    \end{split}
\end{equation}
respectively. By Wente's lemma, we have
\begin{equation}
    \|F^i_j\|_{L^\infty(B_1)}\leq C\|\nabla \phi_1\|_{L^2(B_1)}.
\end{equation}
The assumption $\psi_1\in W^{1,\frac{4s}{4-s}}$ and the Sobolev embedding theorem implies
\begin{equation*}
    \psi_1\in L^{\frac{4s}{4-3s}}(B_1).
\end{equation*}
Since $s>1$, we know
\begin{equation*}
    \frac{4s}{4-s}>4/3, \ \frac{4s}{4-3s}>\frac{4s}{2-s}>4.
\end{equation*}
One can compute $\|{\rm curl}(A^i_k\omega^k_j)\|_{L^s}$ as follows.
\begin{equation*}
\begin{split}
   & \left(\int_{B_1}|{\rm curl}(A^i_k\omega^k_j)|^s\right)^{1/s}\\
   &\leq C \left(\int_{B_1}|\nabla^\perp{A^i_k}\cdot\omega^k_j|^s\right)^{1/s}+C \left(\int_{B_1}|A^i_k{\rm curl}(\omega^k_j)|^s\right)^{1/s}\\
    &\leq C(N)\left(\int_{B_1}|\nabla{A^i_k}|^s|\psi_1|^{2s}\right)^{1/s}
    +C(N)\left(\int_{B_1}|\psi_1|^s|\nabla{\psi_1}|^s\right)^{1/s}.
\end{split}
\end{equation*}
Plugging the following estimates into the inequality above,
\begin{equation*}
\begin{split}
        \left(\int_{B_1}|\nabla{A^i_k}|^s|\psi_1|^{2s}\right)^{1/s}&\leq
            \left(\int_{B_1}|\nabla{A^i_k}|^2\right)^{1/2}    \left(\int_{B_1}(|\psi_1|^{2s})^{\frac{2}{2-s}}\right)^{\frac{1}{s}\frac{2-s}{2}}\\
            &=\|\nabla{A^i_k}\|_{L^2(B_1)}\|\psi_1\|^2_{L^{\frac{4s}{2-s}}(B_1)},
\end{split}
\end{equation*}
\begin{equation*}
\begin{split}
        \left(\int_{B_1}|\psi_1|^s|\nabla{\psi_1}|^s\right)^{1/s}&\leq
            \left(\int_{B_1}|\psi_1|^4\right)^{1/4}    \left(\int_{B_1}(|\nabla\psi_1|^s)^{\frac{4}{4-s}}\right)^{\frac{1}{s}\frac{4-s}{4}}\\
            &=\|\psi_1\|_{L^4(B_1)}\|\nabla\psi_1\|_{L^{\frac{4s}{4-s}}(B_1)},
\end{split}
\end{equation*}
we have
\begin{equation}\label{rhs L^s}
\begin{split}
    &\left(\int_{B_1}|{\rm curl}(A^i_k\omega^k_j)|^s\right)^{1/s}\\
    &\leq C(N)\|\psi_1\|^2_{L^{\frac{4s}{2-s}}(B_1)}\sqrt{\varepsilon_0}+C(N)\|\nabla\psi_1\|_{L^{\frac{4s}{4-s}}(B_1)}\varepsilon_0^{1/4}\\
    &\leq C\left(N,\|\psi_1\|^2_{L^{\frac{4s}{2-s}}(B_1)},\|\nabla\psi_1\|_{L^{\frac{4s}{4-s}}(B_1)}\right)\varepsilon_0^{1/4}.
    \end{split}
\end{equation}

Now, multiplying $G^i_j$ to $\Delta G^i_j$ and using the divergence theorem, we have
  \begin{equation*}
          \int_{B_1}|\nabla G^i_j|^2=-\int_{B_1}G^i_j\Delta G^i_j
          \leq C(N)\int_{B_1}|\nabla \psi_1||\psi_1||G^i_j|+|\psi_1|^2|\nabla A||G^i_j|.
  \end{equation*}
Plugging the following estimates into the inequality above,
\begin{equation*}
\begin{split}
    \int_{B_1}|\nabla \psi_1||\psi_1||G^i_j|&\leq \|\psi_1\|_{L^4(B_1)}\||\nabla{\psi_1}||G^i_j|\|_{L^{4/3}(B_1)}\\
    &\leq\|\psi_1\|_{L^4(B_1)}\|\nabla{\psi_1}\|_{L^{\frac{4s}{4-s}}(B_1)}\|G^i_j\|_{L^{\frac{s}{s-1}}(B_1)}
    \end{split}
\end{equation*}
\begin{equation*}
\begin{split}
    \int_{B_1}|\psi_1|^2|\nabla A||G^i_j|&\leq \|\nabla{A}\|_{L^2(B_1)}\||\psi_1|^2|G^i_j|\|_{L^2(B_1)}\\
    &\leq\|\nabla{A}\|_{L^2(B_1)}\|\psi_1\|_{L^{\frac{4s}{2-s}}(B_1)}^2\|G^i_j\|_{L^{\frac{s}{s-1}}(B_1)},
    \end{split}
\end{equation*}
we have
\begin{equation*}
\begin{split}
          \int_{B_1}|\nabla G^i_j|^2&\leq
          C(N)\left(\|\nabla\psi_1\|_{L^{\frac{4s}{4-s}}(B_1)}\varepsilon_0^{1/4}+\|\psi_1\|^2_{L^{\frac{4s}{2-s}}(B_1)}\sqrt{\varepsilon_0}\right)\|G^i_j\|_{L^{\frac{s}{s-1}}(B_1)}\\
    &\leq C\left(N,\|\psi_1\|^2_{L^{\frac{4s}{2-s}}(B_1)},\|\nabla\psi_1\|_{L^{\frac{4s}{4-s}}(B_1)}\right)\varepsilon_0^{1/4}\|\nabla{G^i_j}\|_{L^2(B_1)},
        \end{split}
  \end{equation*}
where we have used the Poincar{\'e} inequality in the last inequality. The inequality above implies that
  \begin{equation*}
      \|\nabla G^i_j\|_{L^2(B_1)}\leq C\left(N,\|\psi_1\|^2_{L^{\frac{4s}{2-s}}(B_1)},\|\nabla\psi_1\|_{L^{\frac{4s}{4-s}}(B_1)}\right)\varepsilon_0^{1/4}.
  \end{equation*}
  Again, by Poincar{\'e} inequality, we obtain
  \begin{equation}
      \|G^i_j\|_{L^p(B_1)}\leq C\left(N,\|\psi_1\|^2_{L^{\frac{4s}{2-s}}(B_1)},\|\nabla\psi_1\|_{L^{\frac{4s}{4-s}}(B_1)}\right)\varepsilon_0^{1/4},  \ \forall p>1.
  \end{equation}
  Together with \eqref{rhs L^s}, we know
  \begin{equation}
  \|G^i_j\|_{L^\infty(B_1)}\leq C\|G^i_j\|_{W^{2,s}(B_1)}\leq C\left(N,\|\psi_1\|^2_{L^{\frac{4s}{2-s}}(B_1)},\|\nabla\psi_1\|_{L^{\frac{4s}{4-s}}(B_1)}\right)\varepsilon_0^{1/4}.
  \end{equation}
  Hence, the inequality \eqref{B} follows from $B^i_j=F^i_j+G^i_j$ and we complete the proof.

\end{proof}

With such $L^\infty$ norm of $B$ in hands, we can prove the following corollary.
\begin{cor}\label{jacobian}
Under the assumptions of Theorem \ref{energy convexity}, there exists
\begin{equation}
    a\in W^{1,2}(B_1, M_K(\mathbb{R})), \  b\in W^{1,2}(B_1,\mathbb{R}^K)
\end{equation}
such that
\begin{equation}
    \Delta\phi_1=\nabla a\cdot\nabla^\perp b \ \text{in} \ B_1
\end{equation}
and
\begin{equation}
    \|\nabla{a}\|_{L^2(B_1)}+\|\nabla{b}\|_{L^2(B_1)}\leq C\|\nabla\phi_1\|_{L^2(B_1)}.
\end{equation}
In particular, $\phi$ is continuous in $B_1$.
\end{cor}
Although the proof is same as the one in \cite{lamm2013Estimates}, we still write them down for completeness.
\begin{proof}
Equation \eqref{conservation law} and the $L^\infty$ estimates of $A$ and $B$ imply the existence of $\eta\in W^{1,2}(B_1,\mathbb{R}^K)$ such that
\begin{equation*}
    \|\nabla \eta\|_{L^2(B_1)}\leq C\|\nabla \phi_1\|_{L^2(B_1)}
\end{equation*}
and
\begin{equation}\label{perp eta}
    \nabla^\perp\eta=A\nabla \phi_1+B\nabla^\perp\phi_1.
\end{equation}
Multiplying this equation with $A^{-1}$ and taking the divergence yields
\begin{equation*}
    \Delta \phi_1^l=\nabla(A^{-1})^l_k\cdot\nabla^\perp\eta^k-\nabla(A^{-1}B)^l_k\cdot\nabla^\perp\phi_1^k.
\end{equation*}
The continuity of $\phi_1$ follows from the Wente's Lemma.
\end{proof}

\vspace{2em}

\section{Local estimate for the matrix $P$}
In this section, we will show a local estimate for the matrix $P$ by considering the equation for $\Delta P$. Different from the harmonic map case, $\Delta P$ does not have the Jacobian structure. There is a divergence term coursed by the spinor term in $\Omega$.

\begin{lem}\label{eqn for P}
Under the assumptions of Theorem \ref{energy convexity}, there exists
\begin{equation*}
    \xi\in W^{1,2}_0(B_1,so(K)), \ \eta\in W^{1,2}(B_1,\mathbb{R}^K)
\end{equation*}
and
\begin{equation*}
    Q_k, R_k\in W^{1,2}(B_1,{\rm GL}_{K}(\mathbb{R})), \ k=1,\dots,K
\end{equation*}
with
\begin{equation*}
    \|\nabla\xi\|_{L^2(B_1)}+\|\nabla\eta\|_{L^2(B_1)}\leq C\|\nabla\phi_1\|_{L^2(B_1)}
\end{equation*}
and
\begin{equation*}
    \sum_k(\|\nabla Q_k\|_{L^2(B_1)}+\|\nabla R_k\|_{L^2(B_1)})\leq C
\end{equation*}
such that
\begin{equation}
    \Delta P=\nabla{P}\cdot\nabla^\perp\xi+\nabla{Q_k}\cdot\nabla^\perp\eta^k+\nabla{R_k}\cdot\nabla^\perp\phi_1^k-\dive(Z),
\end{equation}
where $Z=\left(Z^i_j\right)_{i,j=1,\dots,K}=\left(\omega^i_mP^m_j\right)_{i,j=1,\dots,K}$.
\end{lem}

\begin{proof}
Multiplying both sides of equation \eqref{xi} by $P$ from the left gives
\begin{equation*}
    \nabla{P}^i_j=P^i_m\nabla^\perp\xi^m_j-\Omega^i_mP^m_j.
\end{equation*}
Taking the divergence on both sides of the equation above yields
\begin{equation}\label{delta p}
\begin{split}
    \Delta P^i_j&=\nabla P^i_m\nabla^\perp\xi^m_j-\dive(\Omega^i_mP^m_j)\\
    &=\nabla P^i_m\nabla^\perp\xi^m_j-\dive(f^i_{ml}\nabla\phi^l_1P^m_j+\omega^i_mP^m_j).
\end{split}
\end{equation}
Multiplying $A^{-1}$ on the both sides of \eqref{perp eta} from the left implies
\begin{equation}\label{nabla phi}
    \nabla\phi_1=A^{-1}\nabla^\perp\eta-A^{-1}B\nabla^\perp\phi_1.
\end{equation}
Plugging \eqref{nabla phi} into \eqref{delta p}, we get
\begin{equation}
    \begin{split}
    \Delta P^i_j
    &=\nabla P^i_m\nabla^\perp\xi^m_j-\dive(f^i_{ml}((A^{-1})^l_k\nabla^\perp\eta^k-(A^{-1}B)^l_k\nabla^\perp\phi_1^k)P^m_j)\\
    &\quad-\dive(\omega^i_mP^m_j)\\
    &=\nabla P^i_m\nabla^\perp\xi^m_j-\nabla(f^i_{ml}(A^{-1})^l_kP^m_j)\cdot\nabla^\perp\eta^k\\
    &\quad+\nabla(f^i_{ml}(A^{-1}B)^l_kP^m_j)\cdot\nabla^\perp\phi_1^k-\dive(\omega^i_mP^m_j).
\end{split}
\end{equation}
Then we complete the proof by denoting
\begin{equation*}
\begin{split}
    (Q_k)^i_j&=-f^i_{ml}(A^{-1})^l_kP^m_j\\
    (R_k)^i_j&=f^i_{ml}(A^{-1}B)^l_kP^m_j\\
    Z^i_j&=\omega^i_mP^m_j.
    \end{split}
\end{equation*}
\end{proof}

Now, let us prove a local estimate on the oscillation of the matrix $P$ based on Lemma \ref{eqn for P}.
\begin{prop}\label{osc p}
Under the assumptions of Theorem \ref{energy convexity}, for any $x\in B_1$, any $r>0$ such that $B_{2r}(x)\subset B_1$ and any $y\in B_r(x)$, we have
\begin{equation}
    |P(y)-P(x)|\leq C\varepsilon_0^{1/4},
\end{equation}
where $C=C(N,\|\psi_1\|_{W^{1,r}(B_1)})$.
\end{prop}

\begin{proof}
Let $P_1\in W^{1,2}(B_1,M_K(\mathbb{R}))$ be the weak solution of \begin{equation*}
\begin{cases}
    \Delta P_1=\nabla{P}\cdot\nabla^\perp\xi+\nabla{Q_k}\cdot\nabla^\perp\eta^k+\nabla{R_k}\cdot\nabla^\perp\phi_1^k \ &\text{in} \ B_1,\\
    P_1=0 \ &\text{on} \ \partial B_1,
    \end{cases}
\end{equation*}
where $Q_k$ and $R_k$ are defined in Lemma \ref{eqn for P}. Then by Wente's Lemma \ref{wente lemma} we have $P_1\in C^0(B_1, M_K(\mathbb{R}))$ and
\begin{equation}\label{P1}
    \|P_1\|_{L^\infty(B_1)}+\|\nabla{P_1}\|_{L^2(B_1)}\leq C\varepsilon_0.
\end{equation}
Consider the solution $P_2\in W^{1,2}(B_1,M_K(\mathbb{R}))$ to
\begin{equation*}
\begin{cases}
    \Delta P_2=-\dive(Z), \ \text{in} \ B_1,\\
    P_2=0, \ \text{on} \ \partial B_1,
        \end{cases}
\end{equation*}
where $Z$ is defined in Lemma \ref{eqn for P} and $Z\in W^{1,1}(B_1,M_K(\mathbb{R})\otimes\wedge^1\mathbb{R}^2)$. One can check this as follows.
\begin{equation*}
\begin{split}
    \|Z\|_{L^2(B_1)}&\leq C\|\psi_1\|_{L^4}^2\leq C\sqrt{\varepsilon_0},\\
    \|Z\|_{W^{1,1}(B_1)}&\leq C\|\psi_1\|_{L^4}^2+\|\nabla Z\|_{L^1}\\
    &\leq C\|\psi_1\|_{L^4}^2+C\|\psi_1\|_{L^4(B_1)}\|\nabla\psi_1\|_{L^{4/3}(B_1)}\\
    &\quad+\|\psi_1\|_{L^4(B_1)}^2\|\nabla P\|_{L^2(B_1)}.
    \end{split}
\end{equation*}

Now, by \cite[Lemma 2.3 and Lemma 2.4]{li20191energy} and \cite[Theorem 3.3.10]{helein2002Harmonic}, we have
\begin{equation}\label{P2}
\begin{split}
     \|P_2\|_{L^\infty(B_1)}
     &\leq C\|\nabla{P_2}\|_{L^{2,1}(B_1)}\\
     &\leq C\|Z\|_{L^{2,1}(B_1)}\\
     &\leq C\|Z\|_{W^{1,1}(B_1)}\\
     &\leq C(N,\|\nabla\psi_1\|_{L^{4/3}(B_1)})\varepsilon_0^{1/4}\\
     &\leq C(N,\|\psi_1\|_{W^{1,r}(B_1)})\varepsilon_0^{1/4}
\end{split}
\end{equation}
and
\begin{equation*}
    \|\nabla{P_2}\|_{L^2}\leq \|Z\|_{L^2(B_1)}\leq C\sqrt{\varepsilon_0}.
\end{equation*}

Since
\begin{equation*}
    \Delta (P-P_1-P_2)=0 \ \text{in} \ B_1,
\end{equation*}
we know $V=P-P_1-P_2$ is harmonic. Therefore, for any $x\in B_1$, any $r>0$ such that $B_{2r(x)}\subset B_1$ and any $y\in B_r(x)$ we have
\begin{equation}
\begin{split}
    &|V(y)-V(x)|\\
    &\leq Cr\|\nabla V\|_{L^\infty(B_r(x))}\\
    &\leq C\|\nabla V\|_{L^2(B_{2r}(x))}\\
    &\leq C\left(\|\nabla{P}\|_{L^2(B_{2r}(x))}+\|\nabla{P_1}\|_{L^2(B_{2r}(x))}+\|\nabla{P_2}\|_{L^2(B_{2r}(x))}\right)\\
    &\leq C\sqrt{\varepsilon_0}.
    \end{split}
\end{equation}
Combining \eqref{P1} and \eqref{P2}, we get
\begin{equation*}
\begin{split}
    |P(y)-P(x)|&=|P_1(y)+P_2(y)+V(y)-P_1(x)-P_2(x)-V(x)|\\
    &\leq |P_1(y)-P_1(x)|+|P_2(y)-P_2(x)|+|V(y)-V(x)|\\
    &\leq C(N,\|\psi_1\|_{W^{1,r}(B_1)})\varepsilon_0^{1/4}
    \end{split}
\end{equation*}
and complete the proof.
\end{proof}

\vspace{2em}

\section{Energy convexity of Dirac-harmonic maps and its corollary}
In this section, we will use the estimates in the previous two sections to prove the energy convexity (Theorem \ref{energy convexity}) and its direct corollary (Corollary \ref{main thm}).

For harmonic maps into general manifolds, the uniqueness follows from the energy convexity:
\begin{equation}\label{convexity}
    \frac12\int_{B_1}|\nabla{\phi_2}-\nabla{\phi_1}|^2\leq \int_{B_1}|\nabla{\phi_2}|^2-\int_{B_1}|\nabla{\phi_1}|^2.
\end{equation}
One can see from the computation below that the constant $\varepsilon_0$ also depends on the constant $c_0$ in the assumption \eqref{1 dh}. Therefore, to a uniform constant $\varepsilon_0$, one has to consider a subset of $W^{1,2}$ space. The reason was stated in the introduction. The Dirac operator in the action functional gives us a curvature term. To control this bad term, we impose the extra condition \eqref{1 dh} on $\phi_2$, which will be only used in \eqref{only place} below.

Let us start to prove \eqref{convexity}. It suffices to show
\begin{equation*}
    \Phi\geq-\frac12\int_{B_1}|\nabla\phi_2-\nabla\phi_1|^2,
\end{equation*}
where
\begin{equation*}
\begin{split}
    \Phi&:=\int_{B_1}|\nabla\phi_2|^2-\int_{B_1}|\nabla\phi_1|^2-\int_{B_1}|\nabla\phi_2-\nabla\phi_1|^2\\
    &=2\int_{B_1}\langle\nabla\phi_2-\nabla\phi_1,\nabla\phi_1\rangle\\
    &=2\int_{B_1}\langle\phi_2-\phi_1,-\Delta\phi_1\rangle\\
    &=-2\int_{B_1}\langle\phi_2-\phi_1,II(d\phi_1,d\phi_1)+\mathcal{R}(\phi_1,\psi_1)\rangle\\
    &\geq-C(N)\int_{B_1}|\phi_2-\phi_1|^\perp|\nabla\phi_1|^2-C(N)\int_{B_1}|\phi_2-\phi_1||\nabla\phi_1||\psi_1|^2.
    \end{split}
\end{equation*}
By \cite[Lemma A.1]{cm2008Width}, H\"older inequality and our assumption \eqref{1 dh}, we have
\begin{equation}\label{only place}
\begin{split}
\Phi&\geq-C(N)\int_{B_1}|\phi_2-\phi_1|^2|\nabla\phi_1|^2-C(N){\varepsilon_0}^{\frac{p}{2}}\left(\int_{B_1}|\phi_2-\phi_1|^2|\nabla\phi_1|^2\right)^{1/2}\\
&\geq-C(N)(1+{\varepsilon_0}^{\frac{p}{2}}c_0^{-1})\int_{B_1}|\phi_2-\phi_1|^2|\nabla\phi_1|^2.
\end{split}
\end{equation}
Note that this is the only place we use the assumption \eqref{1 dh}. To control the integral in the right-hand side of the inequality above, we need the following lemma.

\begin{lem}\label{reduce to varphi}
Under the assumptions of Theorem \ref{main thm}, if we can find a solution $\varphi\in W^{1,2}_0\cap L^\infty(B_1)$ of
\begin{equation}\label{varphi}
    \begin{cases}
        \Delta \varphi=|\nabla\phi_1|^2, \ &\text{in} \ B_1,\\
        \varphi=0 \ &\text{on} \ \partial B_1,
    \end{cases}
\end{equation}
which satisfies
\begin{equation}\label{bound varphi}
    \|\varphi\|_{L^\infty(B_1)}+\|\nabla\varphi\|_{L^2(B_1)}\leq C\int_{B_1}|\nabla\phi_1|^2\leq C\varepsilon_0,
\end{equation}
then we have
\begin{equation}\label{convexity1}
\begin{split}
    \int_{B_1}|\phi_2-\phi_1|^2|\nabla\phi_1|^2&\leq C\int_{B_1}|\nabla \phi_1|^2\int_{B_1}|\nabla\phi_2-\nabla\phi_1|^2\\
    &\leq C\varepsilon_0\int_{B_1}|\nabla\phi_2-\nabla\phi_1|^2.
    \end{split}
\end{equation}
\end{lem}
The proof is same as the one in \cite{colding2007Width} and \cite{lamm2013Estimates}. We just write it down for completeness.
\begin{proof}[Proof of Lemma \ref{reduce to varphi}]
Substituting \eqref{varphi} into the left-hand side of \eqref{convexity1} yields
\begin{equation}\label{stokes}
\begin{split}
    &\int_{B_1}|\phi_2-\phi_1|^2|\nabla\phi_1|^2\\
    &=\int_{B_1}|\phi_2-\phi_1|^2\Delta\varphi\\
    &\leq \int_{B_1}|\nabla|\phi_2-\phi_1|^2||\nabla\varphi|  \\
    &\leq 2\left(\int_{B_1}|\nabla\phi_2-\nabla\phi_1|^2\right)^{1/2}\left(\int_{B_1}|\phi_2-\phi_1|^2|\nabla\varphi|^2\right)^{1/2},
\end{split}
\end{equation}
where we have used Stokes' theorem and the Cauchy-Schwarz inequality. Now applying Stokes' theorem to $\dive(|\phi_2-\phi_1|^2\varphi\nabla\varphi)$ and using $\Delta \varphi\geq0$ and \eqref{stokes},  we have
\begin{equation*}
\begin{split}
    &\int_{B_1}|\phi_2-\phi_1|^2|\nabla\varphi|^2\\
    &\leq \int_{B_1}|\varphi|(|\phi_2-\phi_1|^2\Delta\varphi+|\nabla|\phi_2-\phi_1|^2||\nabla\varphi|)\\
    &\leq 4\|\varphi\|_{L^\infty(B_1)}\left(\int_{B_1}|\nabla\phi_2-\nabla\phi_1|^2\right)^{1/2}\left(\int_{B_1}|\phi_2-\phi_1|^2|\nabla\varphi|^2\right)^{1/2}.
\end{split}
\end{equation*}
Therefore,
\begin{equation}\label{nabla varphi}
\left(\int_{B_1}|\phi_2-\phi_1|^2|\nabla\varphi|^2\right)^{1/2}\leq4\|\varphi\|_{L^\infty(B_1)}\left(\int_{B_1}|\nabla\phi_2-\nabla\phi_1|^2\right)^{1/2}.
\end{equation}
Finally, substituting \eqref{nabla varphi} back into \eqref{stokes} and combining with \eqref{bound varphi} yields
\begin{equation}
\begin{split}
    \int_{B_1}|\phi_2-\phi_1|^2|\nabla\phi_1|^2
    &\leq C\|\varphi\|_{L^\infty(B_1)}\int_{B_1}|\nabla\phi_2-\nabla\phi_1|^2\\
    &\leq C\int_{B_1}|\nabla\phi_1|^2\int_{B_1}|\nabla\phi_2-\nabla\phi_1|^2
\end{split}
\end{equation}
\end{proof}

By Lemma \ref{reduce to varphi}, we have
 \begin{equation}\label{spinor p}
\Phi\geq -C(N)\left(1+{\varepsilon_0}^{\frac{p}{2}}c_0^{-1}\right)\varepsilon_0\int_{B_1}|\nabla\phi_2-\nabla\phi_1|^2.
\end{equation}
Now, we first choose $\varepsilon_0$ such that
\begin{equation}\label{epsilon0}
  C(N)\varepsilon_0\leq\frac14.
\end{equation}
Then, for such $\varepsilon_0$, there exists $p\geq1$ such that
\begin{equation}\label{p}
  c_0= 4C(N)\varepsilon_0^{1+\frac{p}{2}}\leq\varepsilon_0^{\frac{p}{2}}.
\end{equation}

Combing \eqref{spinor p}, \eqref{epsilon0}, \eqref{p} and our assumption \eqref{small energy}, we get the following convexity
\begin{equation*}
    \Phi\geq-\frac12\int_{B_1}|\nabla\phi_2-\nabla\phi_1|^2,
\end{equation*}
and Theorem \ref{energy convexity} and Corollary \ref{main thm} follow under the assumptions in Lemma \ref{reduce to varphi}.

Now, everything boils down to validating the assumptions in Lemma \ref{reduce to varphi}, i.e., the existence of a function $\varphi$ satisfying \eqref{varphi} and \eqref{bound varphi}. In the light of the following regularity theorem (see \cite[Theorem 1.100]{Semmes1994Apoh}, \cite[Theorem 5.1]{chang1993Hp} and \cite[Theorem A.4 ]{lamm2013Estimates}), it will be sufficient to show that $|\nabla\phi_1|^2$ is in the local Hardy space $h^1(B_1)$.

\begin{thm}\cite[Theorem A.4 ]{lamm2013Estimates}
Let $f\in h^1(B_1)$ and assume that $f\geq0$ a.e. in $B_1$. Then there exists a function $v\in L^\infty\cap W^{1,2}(B_1)$ solving the Dirichlet problem
\begin{equation}\label{v}
    \begin{cases}
        \Delta v=f, \ &\text{in} \ B_1,\\
        v=0 \ &\text{on} \ \partial B_1,
    \end{cases}
\end{equation}
Moreover, there exists a constant $C>0$ such that
\begin{equation}
\|v\|_{L^\infty(B_1)}+\|\nabla{v}\|_{L^2(B_1)}\leq C \|f\|_{h^1(B_1)}.
\end{equation}
\end{thm}

Let us recall the definition of the local Hardy space $h^1(B_1)$ before proving $|\nabla\phi_1|^2\in h^1(B_1)$.

\begin{defn}\cite[Definition A.3]{lamm2013Estimates}\label{hardy}
Choose a Schwartz function $\rho\in C^\infty_0(B_1)$ such that
\begin{equation*}
    \int_{B_1}\rho=1
\end{equation*}
and let $\rho_t(x)=t^{-2}\rho(\frac{x}{t})$. For a measurable function $f$ defined in $B_1$ we say that $f$ lies in the local Hardy space $h^1(B_1)$ if the radial maximal function of $f$
\begin{equation*}
\begin{split}
      f^*(x)&=\sup_{0<t<1-|x|}\left|\int_{B_t(x)}\frac{1}{t^2}\rho\left(\frac{x-y}{t}\right)f(y)\right|\\
      &=\sup_{0<t<1-|x|}|\rho*f|(x)
\end{split}
\end{equation*}
belongs to $L^1(B_1)$ and we define
\begin{equation*}
    \|f\|_{h^1(B_1)}=\|f^*\|_{L^1(B_1)}.
\end{equation*}
\end{defn}

Next, we will show $|\nabla\phi_1|^2\in h^1(B_1)$. By using Theorem \ref{xip}, Theorem \ref{AB} and Proposition \ref{L infty B}, for
any $x\in B_1$, any $r>0$ such that $B_{2r}(x)\subset B_1$ and any $y\in B_r(x)$, we have
\begin{equation*}
\begin{split}
     0&\leq\frac12|\nabla\phi_1|^2(y)\leq(A\nabla\phi_1+B\nabla^\perp\phi_1)\cdot(P^T\nabla\phi_1)(y)\\
     &=(A\nabla\phi_1+B\nabla^\perp\phi_1)(y)\cdot\bigg(P^T(x)+P^T(y)-P^T(x)\bigg)\nabla\phi_1(y).
\end{split}
\end{equation*}
Then by Lemma \ref{osc p} and \eqref{perp eta} we obtain
\begin{equation}\label{1/4 energy density}
\begin{split}
    &(\nabla^\perp\eta\cdot P^T(x)\nabla\phi_1)(y)\\
    &=(A\nabla\phi_1+B\nabla^\perp\phi_1)(y)\cdot(P^T(x)\nabla\phi_1)(y)\\
    &\geq\frac12|\nabla\phi_1|^2(y)-(A\nabla\phi_1+B\nabla^\perp\phi_1)(y)\cdot\bigg(P^T(y)-P^T(x)\bigg)\nabla\phi_1(y)\\
    &\geq\frac14|\nabla\phi_1|^2(y).
    \end{split}
\end{equation}
Now, we choose a function $\rho\in C^\infty_0(B_1)$ with
\begin{equation*}
    \rho\geq0, \ \ {\rm spt}(\rho)\subset B_{\frac12}, \ \ \rho(x)|_{B_{\frac38}}=2, \ \ \|\nabla\phi_1\|_{L^\infty(B_1)}\leq100, \ \ \int_{B_1}\rho dx=1.
\end{equation*}
From \eqref{1/4 energy density} and Definition \ref{hardy}, we get
\begin{equation*}
    \begin{split}
        \||\nabla\phi_1|^2\|_{h^1(B_1)}&=\int_{B_1}\sup_{0<t<1-|x|}\rho_t*|\nabla\phi_1|^2dx\\
        &\leq 4\int_{B_1}\sup_{0<t<1-|x|}\rho_t*(\nabla^\perp\eta\cdot P^T(x)\nabla\phi_1)dx\\
        &=4\int_{B_1}\sup_{0<t<1-|x|}\rho_t*\bigg((P^T(x))_{ij}(\nabla^\perp\eta^i\cdot \nabla\phi_1^j)\bigg)dx\\
        &\leq C\|\nabla^\perp\eta\|_{L^2(B_1)}\|\nabla\phi_1\|_{L^2(B_1)}\\
        &\leq C\int_{B_1}|\nabla\phi_1|^2,
    \end{split}
\end{equation*}
where we have used the fact that for all $i,j=1,\dots,K$,
\begin{equation*}
    \nabla^\perp\eta^i\cdot \nabla\phi_1^j\in h^1(B_1), \ \|\nabla^\perp\eta^i\cdot \nabla\phi_1^j\|_{h^1(B_1)}\leq C\|\nabla^\perp\eta\|_{L^2(B_1)}\|\nabla\phi_1\|_{L^2(B_1)},
\end{equation*}
which was proved in \cite{lamm2013Estimates} by the extension argument. For completeness, we write it down as follows.

We first extend
\begin{equation*}
    \eta^i-\frac{1}{|B_1|}\int_{B_1}\eta^i \ \ \text{and} \ \ \phi_1^j-\frac{1}{|B_1|}\int_{B_1}\phi^j
\end{equation*}
from $B_1$ to $\mathbb{R}^2$ which yields the existence of $\tilde{\eta}^i,\tilde{\phi}_1^j\in W^{1,2}_0(\mathbb{R}^2)$ such that
\begin{equation*}
    \int_{\mathbb{R}^2}|\nabla\tilde{\eta}^i|^2\leq C\int_{B_1}|\nabla\eta^i|^2,\ \ \ \ \
    \int_{\mathbb{R}^2}|\nabla\tilde{\phi}_1^j|^2\leq C\int_{B_1}|\nabla\phi_1^j|^2
\end{equation*}
and
\begin{equation}\label{extension}
    \nabla\tilde{\eta}^i=\nabla\eta^i, \ \ \ \ \
    \nabla\tilde{\phi}_1^j=\nabla\phi_1^j \ \ \text{a.e. in} \ B_1.
\end{equation}
Then by \cite[Theorem II.2]{coifman1993Compensated} and \cite{Semmes1994Apoh}, we get
\begin{equation}\label{H^1}
\begin{split}
    &\|\nabla^\perp\tilde{\eta}^i\cdot \nabla\tilde{\phi}_1^j\|_{\mathcal{H}^1(\mathbb{R}^2)}\\
    &:=\int_{\mathbb{R}^2}\sup_{\rho\in\mathcal{T}}\sup_{t>0}\left|\int_{B_t}\frac{1}{t^2}\rho\left(\frac{x-y}{t}\right)(\nabla^\perp\tilde{\eta}^i\cdot \nabla\tilde{\phi}_1^j)(y)dy\right|dx\\
    &\leq C\|\nabla\tilde{\eta}^i\|_{L^2(\mathbb{R}^2)}\|\nabla\tilde{\phi}_1^i\|_{L^2(\mathbb{R}^2)}
    \leq C\|\nabla\eta\|_{L^2(B_1)}\|\nabla\phi_1\|_{L^2(B_1)},
\end{split}
\end{equation}
where $\mathcal{T}:=\{\rho\in C^\infty_0(\mathbb{R}^2): {\rm spt}(\rho)\subset B_1 \ \text{and} \ \|\nabla\rho\|_{L^\infty}\leq 100\}$. From \eqref{extension} and \eqref{H^1}, we obtain
\begin{equation*}
\begin{split}
        \|\nabla^\perp\eta^i\cdot \nabla\phi_1^j\|_{h^1(B_1)}
        &=\|\nabla^\perp\tilde{\eta}^i\cdot \nabla\tilde{\phi}_1^j\|_{h^1(B_1)}\\
        &\leq\|\nabla^\perp\tilde{\eta}^i\cdot \nabla\tilde{\phi}_1^j\|_{\mathcal{H}^1(\mathbb{R}^2)}\\
        &\leq  C\|\nabla\eta\|_{L^2(B_1)}\|\nabla\phi_1\|_{L^2(B_1)}.
\end{split}
\end{equation*}

\section{Proofs of Theorem \ref{unique DH} and Theorem \ref{main uniqueness}}

\begin{proof}[Proof of Theorem \ref{unique DH}]
When $\phi_i\in W^{1,2\alpha}$, we have $|\nabla\phi_i|^2\in h^1(B_1)$ by the Definition \ref{hardy}. Therefore, the proofs of energy convexity and uniqueness of maps are still valid without using the gauge decomposition (i.e. Sections 4 and 5). In other word, $\psi_i\in W^{1,4/3}$ is enough and $\varepsilon_0$ does not depend on $\psi_i$. Hence, we complete the proof of Theorem \ref{unique DH}.

\end{proof}

\begin{proof}[Proof of Theorem \ref{main uniqueness}]
For a given boundary condition, by Theorem \ref{main thm'}, it suffices to exclude the case in \eqref{latter case}. We claim that there is a constant $p_0>1$ such that \eqref{latter case} can not happen. If not, for $p_k\to\infty$, there is a sequence of coupled Dirac-harmonic maps $(\phi_k,\psi_k)$ satisfy
\begin{equation}\label{go to zero}
    E(\phi_k):=\int_{B_1}|\nabla\phi_k|^2\leq\varepsilon_0, \ \ E(\psi_k):=\int_{B_1}|\psi_k|^4\leq\varepsilon_0^{p_k}.
\end{equation}
Then, by compactness \cite{jost2018geometric}, the weak limit of $(\phi_k,\psi_k)$ is a Dirac-harmonic map with same boundary value, denoting by $(\bar\phi,\bar\psi)$. Since $E(\psi_k)$ converges to zero, $\bar\psi$ has to vanish, which contradicts with ${\bf B}\bar\psi={\bf B}\psi_0\not\equiv0$. Therefore, there is a constant $p_0$ such that the uniqueness holds for a given boundary value. 

Suppose the constant $p_0$ depends on the boundary value ${\bf B}\psi_0$, that is, there is a sequence of ${\bf B}\psi_k$ such that $p_k\to \infty$. Again, by the compactness, $(\phi_k,\psi_k)$ weakly converges to the Dirac-harmonic map $(\phi_0,\psi_0)$. Since $E(\psi_k)$ converges to zero, $\psi_0=0$. Therefore, $\phi_0$ is the unique harmonic map with boundary value $\varphi$.

\end{proof}
\vspace{1em}

\vspace{2em}

\bibliographystyle{amsplain}
\bibliography{reference}

\end{document}